\documentclass{amsart}
\usepackage{amsmath, amsthm, amscd, amsfonts, amssymb, color}

\usepackage{amsfonts}
\usepackage{amsmath,amscd}
\usepackage{amssymb}
\usepackage{amsthm}
\usepackage{newlfont}
\newcommand{\f}{\frac}

\long\def\alert#1{\parindent2em\smallskip\hbox to\hsize%
{\hskip\parindent\vrule%
\vbox{\advance\hsize-2\parindent\hrule\smallskip\parindent.4\parindent%
\narrower\noindent#1\smallskip\hrule}\vrule\hfill}\smallskip\parindent0pt}
 \newtheorem{thm}{Theorem}[section]
 \newtheorem{cor}[thm]{Corollary}
 \newtheorem{lem}[thm]{Lemma}
 \newtheorem{prop}[thm]{Proposition}
 \theoremstyle{definition}
 \newtheorem{defn}[thm]{Definition}
 \theoremstyle{remark}

 \numberwithin{equation}{section}

\begin{document}

%
\title[The structure of  generalized Heisenberg Lie algebras]
{The structure, capability and the Schur multiplier of  generalized Heisenberg Lie algebras}

\author[F. Johari]{Farangis Johari}
\address{Department of Pure Mathematics\\
Ferdowsi University of Mashhad, Mashhad, Iran}
\email{farangis.johari@mail.um.ac.ir, farangisjohary@yahoo.com}

\author[P.Niroomand]{Peyman Niroomand}
\address{School of Mathematics and Computer Science\\
Damghan University, Damghan, Iran}
\email{niroomand@du.ac.ir, p$\_$niroomand@yahoo.com}



\date{\today}

\begin{abstract}
From [Problem 1729, Groups of prime power order, Vol. 3], Berkovich et al. asked to obtain the Schur multiplier and the representation of a group $G$,  when $G$ is a special $p$-group minimally generated by $d$ elements and $|G'|=p^{\f{1}{2}d(d-1)}$.
Since there are analogies between groups and Lie
algebras, we intend to give an answer to this question similarly for nilpotent Lie algebras. Furthermore, we give some results about the tensor square and the Schur multiplier of some nilpotent Lie algebras  of class two.

 \end{abstract}

\maketitle

\section{  Motivation and Preliminaries}
The Schur multiplier of groups is appeared in the works of Schur in $1904.$ There are many results on the Schur multiplier of finite $p$-groups and the reader can see for instance \cite{2b2,9,11}. According to \cite{ba,ba1,hardy,ni,ni1}, we may define the Schur multiplier, $\mathcal{M}(L)$,  for a Lie algebra $L$. Many papers in the literature make an attempt to generalize the results on finite $p$-groups to the theory of Lie algebras. Although there are some sporadic results  for the Lie algebra that does not coincide with the results for groups. However, there are analogies between groups and Lie
algebras, but the analogies  are not completely identical and  most of them should be checked carefully.
Recently in \cite{rai}, Rai tried to give an answer to \cite[Problem 1729]{berk}. The motivation of this paper is not only  to answer this question for Lie algebra but we also determine which one of these Lie algebras are capable (a Lie algebra $L$ is called capable if and only if $L\cong E/Z(E)$ for a Lie algebra $E$). By $d(L)$ we denote the minimal number of elements required to generate a Lie algebra $L.$ Roughly speaking, for a generalized Heisenberg  Lie algebra $H,$ $H^2=Z(H),$ such that $d(H)=d$  and $\dim H^2=\f{1}{2}d(d-1),$ we intend to describe the structure of $H$ and its Schur multiplier and we also specify when $H$ is capable. 
%
In this paper, we are going to consider the  generalized Heisenberg  Lie algebra $H$ with the maximum  dimension of the derived subalgebra, that means, if  $d(H)=d,  $ then $\dim H^2=\f{1}{2}d(d-1).  $
Our technique depends on the results that recently obtained in \cite{ni4} and completely is different from \cite{rai}.
Furthermore, we give the structure of Schur
multiplier and tensor square of  Lie algebras of class two with the maximum  dimension of the derived subalgebra as an application and we show such Lie algebras are capable. 
Throughout the paper, $H(m)$ and $A(n)$ are used to denote the Heisenberg and abelian Lie algebra of dimension $2m+1$ and $n,$ respectively.  For the convenience of the reader, we give some results and definitions in the following.

\begin{lem}\cite[Lemma 2.6]{ni3}\label{2}
We have
\begin{itemize}
\item[$(i)$]$\dim \mathcal{M}(A(n))=\dfrac{1}{2}n(n-1).$
\item[$(ii)$]$\dim \mathcal{M}(H(1))=2.$
\item[$(iii)$]$\dim \mathcal{M}(H(m))=2m^2-m-1$ for all $ m\geq 2.$
\end{itemize}
\end{lem}

For a Lie algebra $L,$ we use notation $L^{(ab)}$ instead of $L/L^2$.
\begin{thm}\label{1kg}
 \cite[Proposition 3]{ba1}
Let $ A$ and $ B $ be two Lie algebras. Then
\[ \mathcal{M}(A\oplus B)\cong   \mathcal{M}(A) \oplus  \mathcal{M}(B) \oplus (A^{(ab)}\otimes_{mod} B^{(ab)}), \] in where $A^{(ab)}\otimes_{mod} B^{(ab)}$ is the standard tensor product $A$ and $B$.
\end{thm}

From \cite{ellis}, let $L\wedge L$ and $L\otimes L$ denote the exterior square and the tensor square of a Lie algebra $L$, respectively.
The authors assume that the reader is some what familiar with  the exterior square and the tensor square of a Lie algebra. See for instance \cite{ellis,ni5}.\newline Recall from \cite{ni3} that the concept exterior centre of a Lie algbera $L$ is denoted by $Z^{\wedge}(L)$ and it is equal to the set $\{x\in L~|~ x\wedge g=1_{L\wedge L}~\text{for all}~ g\in L \}$. It is well-known that $L$ is capable if and only if $Z^{\wedge}(L)=0$.

The next lemma  shows the kernel of  the commutator map  $ \kappa' $ is exactly the Schur multiplier.
\begin{lem}\cite[Theorem 35 $(iii)$]{ellis}\label{j1}
Let $ L $ be a Lie algebra. Then
$0\rightarrow \mathcal{M}(L)\rightarrow L\wedge L \xrightarrow{\kappa'} L^2\rightarrow 0$ is exact, in where
$\kappa': L\wedge L \rightarrow L^2$ is given by $l\wedge l_1\mapsto [l,l_1].$

\end{lem}

In the following, we provide a criterion  for detecting the
capability of a Lie algebra $L.$
\begin{lem}\cite[Corollary 4.6]{alam}\label{a} A Lie algebra $L$ is capable if and only if the natural map $\mathcal{M}(L)\rightarrow  \mathcal{M}(L/\langle x\rangle)$ has a non-trivial kernel for all non-zero elements $x \in Z(L).$
\end{lem}

Recall that a Lie algebra $ L $ is called Heisenberg provided that $ L^2=Z(L) $ and $ \dim L^2=1.$ Such algebras are odd dimension and have the following presentation $ H(m)\cong \langle a_1,b_1,\ldots, a_m,b_m,z\big{|}[a_l,b_l]=z,1\leq
l\leq m\rangle.$

\begin{defn}\label{1}
A Lie algebra $ H $ is called generalized Heisenberg of rank $ n $ if $ H^2=Z(H) $ and $ \dim H^2=n $.
\end{defn}
The following  result allows us to work only on the generalized Heisenberg Lie algebras when working on the capability of  Lie algebras of class $2. $

\begin{prop}\cite[Proposition 2.2 ]{ni4} and \cite[Proposition 3.1]{C} \label{811}
Let $ L $ be a finite dimensional nilpotent Lie algebra of nilpotency class $ 2 $. Then $ L=H\oplus A $ and $ Z^{\wedge}(L)=Z^{\wedge}(H)$, where $ A $ is abelian  and $ H $ is a generalized Heisenberg Lie algebra.
\end{prop}
Let $L$ be a free Lie algebra  on the set $X=\{x_1,x_2,\ldots \}$ From \cite{shi}, the  basic commutator on the set $X$ defined inductively.
\begin{itemize}
\item[(i)] The generators $x_1,x_2,\ldots, x_n$ are basic commutators of length one and ordered by setting $x_i < x_j$ if $i < j.$

\item[(ii)] If all the basic commutators $d_i$ of length less than $t$ have been defined and ordered, then we may define the basic commutators of length $t$ to be all commutators of the form $[d_i, d_j]$ such that the sum of lengths of $d_i$ and $d_j$ is $t$, $d_i > d_j$, and if $d_i =[d_s, d_t]$, then $d_j\geq d_t$. The basic commutators of length $t$ follow those of lengths less than $t$. The basic commutators of the same length can be ordered in any way, but usually the lexicographical order is used.
\end{itemize}

   The number of all basic commutators on a set $X=\{x_1,x_2,\ldots x_d\}$ of length $n$ is denoted by $l_d(n)$. Thanks to \cite{m},  we have
   \[l_d(n)=\frac{1}{n}\sum_{m|n}\mu (m)d^{\f{n}{m}},\]
   where $\mu (m)$ is the M\"{o}bius function, defined by $\mu (1) = 1, \mu (k) = 0$ if $k$ is divisible by a square, and
$\mu (p_1 \ldots p_s) = (-1)^s $ if $p_1,\ldots , p_s$ are distinct prime numbers.\\

Using the the topside statement and looking  \cite[Lemma 1.1]{sal} and \cite{shi}, we have the following.\newline
\begin{thm}\label{13}
Let $ F $ be a free Lie algebra on set $ X $, then $ F^c/ F^{c+i}$ is an abelian Lie algebra with the basis of all basic commutators on $ X $ of lengths $ c,c+1,\ldots,c-i+1 $ for all $0 \leq i \leq c$. In particular, $ F^c/ F^{c+1}$ is an abelian Lie algebra of dimension $l_d(c)$, where $ F^{c} $ is the $c$-th term of the lower central series of $F$.
\end{thm}

\section{Main results}
In this section, we intend to obtain the structure of a generalized Heisenberg Lie algebra $ H $ of rank $\f{1}{2}d(d-1),$ in where $d=d(H).$ Furthermore, we describe the Schur multiplier of such Lie algebras and then we detect  which ones of them are capable. It gives an  affirmative answer to \cite[Problem 1729]{berk} and generalized and enriched the result of \cite{rai} with a quite different way.

The following proposition determines  the minimal number of elements required to generate a finite dimensional nilpotent Lie algebra. 
\begin{prop}\label{1h}
Let L be a $d$-generator nilpotent Lie algebra of dimension n with the derived subalgebra of dimension $m. $ Then  $d=\dim L/L^2=n-m. $
\end{prop}
\begin{proof}
We can choose a basis set $ \{ x_1+ L^2,\ldots,x_{n-m}+ L^2 \} $ for  $ L/L^2. $ Thus  $  L=\langle  x_1,\ldots,x_{n-m}\rangle +L^2.$ By \cite[Corollary 2]{mar}, $ L^2 $ is equal to the intersection of all maximal subalgebras of $L.$ Now \cite[Lemma 2.1]{two} implies   $ L=\langle  x_1,\ldots,x_{n-m}\rangle$ and so $ d\leq n-m.$ Let now $ L=\langle  y_1,\ldots,y_{d}\rangle.$ Therefore $ L/L^2=\langle  y_1+L^2,\ldots,y_{d}+L^2\rangle. $ Thus we have $  \dim L/L^2=n-m \leq d.$ Hence $ d=\dim L/L^2=n-m,$ as required.
\end{proof}
We need the following easy lemma.

\begin{lem}\cite[Lemma 14]{mon}\label{kk}
 Let $L$ be a  Lie algebra such that $ \dim (L/Z(L))=n. $ Then
 $ \dim L^2\leq \frac{1}{2} n(n-1).$
\end{lem}
 
\begin{cor}\label{hh}
Let $H$ be an $n$-dimensional generalized Heisenberg Lie algebra of rank $m.$ Then  $d(H)=\dim (H/Z(H))=n-m $ and
 $ \dim H^2\leq \frac{1}{2} (n-m)(n-m-1).$
\end{cor}
\begin{proof}
It is a conclusion of  Proposition \ref{1h} and Lemma \ref{kk}.
\end{proof}
We are a position to determine the structure a $d$-generator generalized Heisenberg Lie algebra $ H $ of rank $ \frac{1}{2}d(d-1).$ 
\begin{prop}\label{h1}
 Let $ H $ be a  $d$-generator generalized Heisenberg Lie algebra of rank $ \frac{1}{2}d(d-1).$ Then  $ \dim H=\frac{1}{2}d(d+1) $ and $ H $ has the presentation $ \langle x_1,\ldots,x_d,y_{i_j}|[x_i,x_j]$  $=y_{i_j},1\leq i<j\leq d\rangle. $
\end{prop}
\begin{proof} 
By Corollary \ref{hh}, we have  $ d(H)=\dim (H/Z(H))= d.$
We can choose a basis set $ \{ x_1+ Z(H),\ldots,x_d+ Z(H)\} $ for the $ H/Z(H) $ such that $ [x_t,x_r] $ is non-trivial for all $ 1\leq t,r\leq d $ and $ t\neq r. $
It is clear   to see that the set $ \{[x_i,x_j]|1\leq i<j\leq d\} $ is a basis of $ H^2. $ Since $ \dim H^2= \frac{1}{2}d(d-1),$ we have 
$H= \langle x_1,\ldots,x_d,y_{i_j}|[x_i,x_j]=y_{i_j},1\leq i<j\leq d\rangle $ and $ \dim H=\frac{1}{2}d(d+1).$ The result follows.
\end{proof}
In the following, we give the structure of the Schur multiplier of a $d$-generator generalized Heisenberg Lie algebra $ H $ of rank $ \frac{1}{2}d(d-1).$ 
\begin{prop}\label{h2}
Let $ H $ be a  $d$-generator generalized Heisenberg Lie algebra of rank $ \frac{1}{2}d(d-1).$ Then
$  \mathcal{M}(H)\cong A(l_d(3)). $
\end{prop}
\begin{proof}
By Proposition \ref{h1}, we have $ H\cong \langle x_1,\ldots,x_d,y_{i_j}|[x_i,x_j]=y_{i_j},1\leq i<j\leq d\rangle. $
Using the method of Hardy and Stitzinger in \cite{hardy}, since $\mathcal{M}(H)$ is abelian,  we just need to compute $\dim \mathcal{M}(H). $ Start with
\begin{align*}
&[x_i,x_j]=y_{i_j}+s_{i_j},1\leq i<j\leq d,\\
&[y_{i_j},x_t]=z_{i_{j}t},1\leq i<j\leq d~\text{and}~1\leq t\leq d,\\
&[y_{i_j},y_{k_t}]=a_{i_{j}k_{t}},1\leq i<j\leq d~\text{and}~1\leq k<t\leq d,
\end{align*}
where $s_{i_j},z_{i_{j}t},$ and $a_{i_{j}k_{t}}  $ generate $ \mathcal{M}(H). $ Putting $y_{i_j}'=y_{i_j}+s_{i_j},1\leq i<j\leq d.  $
A change of variables allows that $ s_{i_j}=0,$ for all $1\leq i<j\leq d. $ Using of the Jacobi identity, we have $a_{i_{j}k_{t}}=[y_{i_j},y_{k_t}]=[x_i,x_j,y_{k_t}]=[z_{i_{j}t},y_{k_t}]=0, $ for all $ 1\leq i<j\leq d$ and $1\leq k<t\leq d. $ We know that $z_{i_{j}t}=[y_{i_j},x_t]$ for all $1\leq i<j\leq d,$ and $1\leq t\leq d$ is a simple commutator of the length three . Put $ S=\langle z_{i_{j}t}|1\leq i<j\leq d~\text{and}~1\leq t\leq d\rangle. $ Clearly, the set of all simple basic commutators of  length three is a basis of $ S. $ Thus $ \mathcal{M}(H)=S $ and $\dim S=l_d(3),$ by using Theorem \ref{13}. The proof is completed.
\end{proof}
We are ready to decide about the capability of a  $d$-generator generalized Heisenberg Lie algebra $ H $ of rank $ \frac{1}{2}d(d-1).$ At first, we need the following proposition.
\begin{prop}\label{h3}
Let $ H$ be a  $d$-generator generalized Heisenberg Lie algebra of rank $ \frac{1}{2}d(d-1)$ such that  $K$ be a one-dimensional subalgebra containing in $ H^2.$ Then $ \mathcal{M}(H/K)\cong A(\l_d(3)-d+1).$ 
\end{prop}
\begin{proof}
By Proposition \ref{h1}, we have $H\cong \langle x_1,\ldots,x_d,y_{i_j}|[x_i,x_j]=y_{i_j},1\leq i<j\leq d\rangle. $ Clearly $ K=\langle [x_i,x_j] \rangle$ for all $ 1\leq i<j\leq d $. Therefore $H/K\cong  \langle x_1+K,\ldots,x_d+K,y_{l_m}+K|[x_l,x_m]+K=y_{l_m}+K,1\leq l<m\leq d,~\text{and}~,l\neq i,m\neq j\rangle$.
Using the method of Hardy and Stitzinger in \cite{hardy}, we compute $\dim \mathcal{M}(H/K). $ Start with
\begin{align*}
&[x_l,x_m]=y_{l_m}+s_{l_m}
,1\leq l<m\leq d~\text{for}~l\neq i ~\text{and}~ m\neq j,\\&[x_i,x_j]=s_{i_j},  \\
&[y_{l_m},x_t]=z_{l_{m}t},1\leq l<m\leq d ~\text{and}~1\leq t\leq d ~\text{for}~l\neq i ~\text{and}~ m\neq j,\\
&[y_{l_m},y_{k_t}]=a_{l_{m}k_{t}},1\leq l<m\leq d~(l\neq i ~\text{and}~m\neq j)~\text{and}~1\leq k<t\leq d~\\&(k\neq i~\text{and}~t\neq j)
\end{align*}
where $s_{l_m},s_{i_j},z_{l_{m}t}  $ and $ a_{l_{m}k_{t}} $ generate $ \mathcal{M}(H/K). $ Putting $y_{l_m}'=y_{l_m}+s_{l_m},1\leq l<m\leq d,~\text{for}~l\neq i~\text{and}~m\neq j.  $ 
A change of variables allows that $ s_{l_m}=0, $ for all $1\leq l<m\leq d $($l\neq i$ and $m\neq j$). Use of the Jacobi identity,  $a_{l_{m} k_{t}}=[y_{l_{m}},y_{k_t}]=[x_l,x_m,y_{k_t}]=[z_{l_{m}t},y_{k_t}]=0 $ for all $1\leq l<m\leq d$~ ( $l\neq i$ and $m\neq j$) and $1\leq k<t\leq d$~ (  $k\neq i$ and $t\neq j$). Thus the set $\{z_{l_m t},s_{i_j}|1\leq l<m\leq d~\text{and}~1\leq t\leq d~\text{for}~l\neq i ~\text{and}~m\neq j\}$ generates $ \mathcal{M}(H/K).  $ We know that $z_{p_qt}=[y_{p_q},x_t]$   is a simple commutator of length  three for all $1\leq p<q\leq d$ and $1\leq t\leq d.$ Put $ S=\langle z_{p_qt}|1\leq p<q\leq d ~\text{and }~1\leq t\leq d\rangle. $  By the definition of basic commutators, the generators $x_1,x_2,\ldots, x_d$ are basic commutators of length one and ordered by setting $x_i < x_j$ if $i < j.$ Clearly, the set of all simple basic commutators of  the length three is a basis of $ S. $ Thus $ S=\langle [x_j,x_i,x_t]|1\leq i<j\leq d~\text{and}~i\leq t\leq d\rangle $ and $\dim S=l_d(3),$ by using Theorem \ref{13}.
Since $[s_{i_j},x_k]=[x_i,x_j,x_k]=0,  $ we have $ [x_j,x_i,x_k]\notin \mathcal{M}(H/K)$  for all $ i\leq k\leq d. $ By using Jacobian identity, we have $[x_j,x_r,x_i]=[x_i,x_r,x_j]  $ for all $ 1\leq r\leq i-1.  $ Thus  the set $ (\{ [x_j,x_i,x_t]|1\leq i<j\leq d~\text{and}~i\leq t\leq d\} - \{ [x_j,x_i,x_k],[x_j,x_r,x_i]| i\leq k \leq d~\text{and}~1\leq r\leq i-1 \})\cup \{s_{i_j}\} $ is a basis of $\mathcal{M}(H/K)  $ and so  $ \dim  \mathcal{M}(H/K)=\l_d(3)-d+1, $ as required.
 \end{proof}
The following theorem shows all such Lie algebras are capable.
 \begin{thm}\label{h6}
Let $ H $ be a  $d$-generator generalized Heisenberg Lie algebra of rank $ \frac{1}{2}d(d-1).$ Then $ H $ is capable.
\end{thm}
\begin{proof}
Propositions \ref{h2} and \ref{h3} imply $ \dim  \mathcal{M}(H)=\l_d(3)$ and
 $ \dim  \mathcal{M}(H/K)=\l_d(3)-d+1 $  for every one-dimensional central ideal $  K $ of $ H. $ Since $ \dim  \mathcal{M}(H)>$\newline$ \dim  \mathcal{M}(H/K),  $ the homomorphism $    \mathcal{M}(H)\rightarrow  \mathcal{M}(H/K) $ is not monomorphism. Thus the result follows from Lemma \ref{a}.
\end{proof}
Now we show that the converse of Proposition \ref{h2} is also true.

\begin{thm}
Let $ H $ be a $d$-generator generalized Heisenberg Lie algebra. Then $\dim H^2=\frac{1}{2}d(d-1)$ if and only if
$  \mathcal{M}(H)\cong A(l_d(3)).$
\end{thm}
\begin{proof}
Let $  \mathcal{M}(H)\cong A(l_d(3)).$ By a similar to the proof of Proposition \ref{h3}, we can see that if $\dim H^2<\frac{1}{2}d(d-1),$ then $\dim  \mathcal{M}(H)< l_d(3).$ Thus $\dim H^2=\frac{1}{2}d(d-1).$ The converse holds by  Proposition \ref{h2}.
\end{proof}
Recall that a pair of Lie algebras $(K, M)$ is said to be a defining pair for $L$ if  $ M\subseteq Z(K) \cap K^2 $ and $K/M\cong L. $ When $L $ is finite-dimensional then the dimension of $K$ is bounded.
If $(K, M)$ is a defining pair for $L,$ then a $K$ of maximal dimension
is called a cover for $L.$ Moreover, from \cite{ba,mon}, in this case $M\cong \mathcal{M}(L).$
\begin{thm}\label{h5}
Let $ H$ be a  $d$-generator generalized Heisenberg Lie algebra of rank $ \frac{1}{2}d(d-1).$  Then $ H^*$  is covering of $ H $ if and only if $ H^*$  is nilpotent of class $3,$  $ \dim (H^*)^3= \dim \mathcal{M}(H), $ $ Z(H^*) \subseteq (H^*)^2$
and $H^*/(H^*)^3\cong H.$\end{thm}
\begin{proof}
Let $ H^*$ be a covering Lie algebra of $ H. $ Then there exists an ideal $ B $ of  $ H^*$ such that $ B\cong \mathcal{M}(H), $  $ B\subseteq Z(H^*) \cap (H^*)^2 $ and $ H^*/B\cong H. $ Since $ Z(H^*/B)=(H^*/B)^2=(H^*)^2/B, $ we have $ Z(H^*) \subseteq (H^*)^2.  $ We claim that $ H^*$  is nilpotent of class $3.$ Clearly $ H^* $ is not abelian since $ H $ is non-abelian. By contrary, let $ H^*$  is nilpotent of class $2.$  Since $ B\subseteq (H^*)^2 =Z(H^*), $ we have  \[ d=\dim H/Z(H)=\dim H/H^2=\]\[\dim (H^*/B)/((H^*/B)^2)= \dim (H^*/(H^*)^2)=\dim (H^*/Z(H^*)).\]By Lemma \ref{kk}, we have $ \dim  (H^*)^2\leq \frac{1}{2}d(d-1).$ Now  Proposition \ref{h2} shows $\dim B=\dim \mathcal{M}(H)=\l_d(3)\leq \dim (H^*)^2 \leq \frac{1}{2}d(d-1)= \l_d(2). $ It is a contradiction. Thus $ H^*$  is nilpotent of class $3$ and so $ (H^*)^3\subseteq B $ and $ \dim (H^*)^2-\dim B= \dim H^2. $ Since $ H^*/(H^*)^3 $ is of class $ 2 $ and $\dim (H^*/(H^*)^3 )^{ab}=d,  $ we have $ \dim ((H^*)^2/(H^*)^3 ) \leq \frac{1}{2}d(d-1),$ by Lemma \ref{kk}. Therefore 
\begin{align*}
&\frac{1}{2}d(d-1)+\dim B-\dim (H^*)^3=\dim H^2+\dim B-\dim (H^*)^3=\\&\dim (H^*)^2-\dim(H^*)^3  =\dim ((H^*)^2/(H^*)^3 ) \leq \frac{1}{2}d(d-1).
\end{align*}
Hence $ \dim B \leq \dim (H^*)^3 $ and $ (H^*)^3\subseteq B. $
Thus $ B=(H^*)^3. $ The converse is clear.
\end{proof}

We are going to characterize the structure of the Schur multiplier and tensor square of
a Lie algebra $ L $ of class two such that   $ \dim (L/Z(L))= d$ and $ \dim L^2= \frac{1}{2}d(d-1).$ 

\begin{thm}\label{h4}
 Let $L$ be an $n$-dimensional Lie algebra of class two such that   $ \dim (L/Z(L))= d$ and $ \dim L^2= \frac{1}{2}d(d-1).$ Then  $ L=H\oplus A(t) $ for some $ t\geq 0 $ such that  $H\cong \langle x_1,\ldots,x_d,y_{i_j}|[x_i,x_j]=y_{i_j},1\leq i<j\leq d\rangle,$  $Z(H)=H^2=L^2\cong A(\frac{1}{2}d(d-1)),$ $ n=\frac{1}{2}d(d+1)+t, $ and $ \mathcal{M}(L)\cong A(l_d(3)+l_t(2)+dt). $ 
\end{thm}
\begin{proof}
Propositions \ref{811} and  \ref{h1} imply  $ L=H\oplus A(t), $ where  $H\cong \langle x_1,\ldots,x_d,y_{i_j}$  $|[x_i,x_j]=y_{i_j},1\leq i<j\leq d\rangle$ and  $Z(H)=H^2=L^2\cong A(\frac{1}{2}d(d-1)) $ for some $ t\geq 0. $
By using Lemma \ref{2} $ (i), $ Theorem \ref{1kg} and Proposition \ref{h2}, we have \begin{align*}
\dim \mathcal{M}(L)&=\dim \mathcal{M}(H)+\dim \mathcal{M}(A(t))+\dim ((H )^{(ab)} \otimes_{mod} A(t))=\\&l_d(3)+\dfrac{1}{2}t(t-1)+dt=l_d(3)+l_t(2)+dt.
\end{align*}
 Hence the result follows.
\end{proof}
\begin{cor}\label{fg}
Let $L$ be an $n$-dimensional Lie algebra of class two such that   $ \dim (L/Z(L))= d$ and $ \dim L^2= \frac{1}{2}d(d-1).$ Then $ L\wedge L\cong \mathcal{M}(L)\oplus L^2\cong A(l_d(3)+l_t(2)+dt+l_d(2)) $ such that $ n=\frac{1}{2}d(d+1)+t $ for some $ t\geq 0. $
\end{cor}
\begin{proof}
Theorem \ref{h4} implies that 
 $L\cong \langle x_1,\ldots,x_d,y_{i_j}|[x_i,x_j]=y_{i_j},1\leq i<j\leq d\rangle \oplus A(t) $ for some $ t\geq 0. $ Since \[[l_1\wedge l_2, l_3\wedge l_4]=[l_1, l_2]\wedge [l_4,l_3] =l_1\wedge [l_2, [l_4, l_3]]-l_2\wedge [l_1, [l_4, l_3]]=0,\] for all $ l_1,l_2, l_3, l_4 \in L,$ we have $ (L\wedge L)^2=0. $ Thus $ L\wedge L\cong \mathcal{M}(L)\oplus L^2\cong A(l_d(3)+l_t(2)+dt)\oplus A(\frac{1}{2}d(d-1)) \cong A(l_d(3)+l_t(2)+dt+l_d(2)) $ such that $ n=\frac{1}{2}d(d+1)+t $ for some $ t\geq 0, $ by using Lemma \ref{j1} and Theorem \ref{h4}.
\end{proof}
Recall that $L\square L\cong \langle l\otimes l~|~ l\in L\rangle.$ Using \cite[Theorem 2.5]{ni5}, we have
\begin{cor}
Let $L$ be an $n$-dimensional Lie algebra of class two such that   $ \dim (L/Z(L))= d$ and $ \dim L^2= \frac{1}{2}d(d-1).$
Then \[ L\otimes L\cong \mathcal{M}(L)\oplus L^2\oplus L\square L \cong   \mathcal{M}(L)\oplus L^2\oplus (L/L^2\square L/L^2)\]\[\cong A(l_d(3)+l_t(2)+dt+l_d(2)+\frac{1}{2}(d+t)(d+t+1))\] for some $ t\geq 0. $
\end{cor}
\begin{proof}
By Corollary \ref{fg}, \cite[Lemmas 2.2 and 2.3]{ni5}, we have $L/L^2\square L/L^2\cong A(\frac{1}{2}(d+t)(d+t+1)) $ and so \[L\otimes L \cong A(l_d(3)+l_t(2)+dt)\oplus A(\frac{1}{2}d(d-1))\oplus  A(\frac{1}{2}(d+t)(d+t+1))\]\[\cong A(l_d(3)+l_t(2)+dt+l_d(2)+\frac{1}{2}(d+t)(d+t+1)),  \] as required.
\end{proof}

\begin{cor}
Let $L$ be an $n$-dimensional Lie algebra of class two such that   $ \dim (L/Z(L))= d$ and $ \dim L^2= \frac{1}{2}d(d-1).$ Then $ L $ is capable.
\end{cor}
\begin{proof}
The result follows from Proposition \ref{811}, Theorems \ref{h6} and \ref{h4}.
\end{proof}

\end{document}